\documentclass[english,oneside]{smfart}

\newcommand{\opn}[1]{\operatorname{#1}}

\usepackage{amsthm,mathtools,amssymb,textcomp,extarrows,bm,mleftright,graphicx,stmaryrd,scalerel}
\mleftright

\makeatletter
\AtBeginDocument{%
 \let\glb@currsize\relax
}
\makeatother

\usepackage{mathrsfs}

\ExplSyntaxOn
\NewDocumentCommand{\definealphabet}{mmmm}
 {
  \int_step_inline:nnn { `#3 } { `#4 }
   {
    \cs_new_protected:cpx { #1 \char_generate:nn { ##1 }{ 11 } }
     {
      \exp_not:N #2 { \char_generate:nn { ##1 } { 11 } }
     }
   }
 }
\ExplSyntaxOff
\definealphabet{bb}{\mathbb}{A}{Z}
\definealphabet{cal}{\mathcal}{A}{Z}
\definealphabet{frak}{\mathfrak}{A}{z}
\definealphabet{rm}{\mathrm}{A}{z}
\definealphabet{bf}{\mathbf}{A}{Z}
\definealphabet{scr}{\mathscr}{A}{Z}
\definealphabet{tt}{\mathtt}{A}{Z}

\newcommand{\lto}{\longrightarrow}

\newcommand{\pparen}[1]{\left(\!\left(#1\right)\!\right)}
\newcommand{\bbrac}[1]{\llbracket #1 \rrbracket}

\usepackage[pdfencoding=auto,psdextra]{hyperref}
\usepackage[nameinlink]{cleveref}
\Crefname{theorem}{Theorem}{Theorem}
\Crefname{conjecture}{Conjecture}{Conjectures}
\Crefname{lemma}{Lemma}{Lemmas}
\Crefname{definition}{Definition}{Definitions}
\Crefname{remark}{Remark}{Remarks}
\Crefname{proposition}{Proposition}{Propositions}
\Crefname{corollary}{Corollary}{Corollaries}
\Crefname{equation}{}{}
\Crefname{item}{}{}
\Crefname{example}{Example}{Examples}
\Crefname{proof}{Proof}{Proofs}

\usepackage[url=true,backend=biber,style=ext-alphabetic,hyperref=true,giveninits=true]{biblatex}
\makeatletter
\renewcommand*\subjclass[2][2020]{%
  \def\@subjclass{#2}%
  \@ifundefined{subjclassname@#1}{%
    \ClassWarning{\@classname}{Unknown edition (#1) of Mathematics
      Subject Classification; using '2020'.}%
  }{%
    \@xp\let\@xp\subjclassname\csname subjclassname@#1\endcsname
  }%
}
\@namedef{subjclassname@1991}{%
  \textup{1991} Mathematics Subject Classification}
\@namedef{subjclassname@2000}{%
  \textup{2000} Mathematics Subject Classification}
\@namedef{subjclassname@2010}{%
  \textup{2010} Mathematics Subject Classification}
\@namedef{subjclassname@2020}{%
  \textup{2020} Mathematics Subject Classification}
\@xp\let\@xp\subjclassname\csname subjclassname@2020\endcsname
\makeatother
\usepackage{framed}
\newtheorem{theorem}{Theorem}[section]
\newtheorem{example}[theorem]{Example}
\AddToHook{env/example/begin}{\crefalias{theorem}{example}}
\newtheorem{lemma}[theorem]{Lemma}
\AddToHook{env/lemma/begin}{\crefalias{theorem}{lemma}}
\newtheorem{remark}[theorem]{Remark}
\AddToHook{env/remark/begin}{\crefalias{theorem}{remark}}
\newtheorem{proposition}[theorem]{Proposition}
\AddToHook{env/proposition/begin}{\crefalias{theorem}{proposition}}
\newtheorem{definition}[theorem]{Definition}
\AddToHook{env/definition/begin}{\crefalias{theorem}{definition}}
\newtheorem{conjecture}[theorem]{Conjecture}
\AddToHook{env/conjecture/begin}{\crefalias{theorem}{conjecture}}
\newtheorem{corollary}[theorem]{Corollary}
\AddToHook{env/corollary/begin}{\crefalias{theorem}{corollary}}

\AddToHook{env/question/begin}{\crefalias{theorem}{question}}

\AddToHook{env/assumption/begin}{\crefalias{theorem}{assumption}}

\newtheorem{introthm}{Theorem}

\newtheorem{introconj}[introthm]{Conjecture}

\AddToHook{env/theorem/before}{\addvspace{3pt}}
\AddToHook{env/theorem/after}{\addvspace{3pt}}
\AddToHook{env/example/before}{\addvspace{3pt}}
\AddToHook{env/example/after}{\addvspace{3pt}}
\AddToHook{env/definition/before}{\addvspace{3pt}}
\AddToHook{env/definition/after}{\addvspace{3pt}}
\AddToHook{env/proposition/before}{\addvspace{3pt}}
\AddToHook{env/proposition/after}{\addvspace{3pt}}
\AddToHook{env/lemma/before}{\addvspace{3pt}}
\AddToHook{env/lemma/after}{\addvspace{3pt}}
\AddToHook{env/conjecture/before}{\addvspace{3pt}}
\AddToHook{env/conjecture/after}{\addvspace{3pt}}
\AddToHook{env/introthm/before}{\addvspace{3pt}}
\AddToHook{env/introthm/after}{\addvspace{3pt}}
\AddToHook{env/introconj/before}{\addvspace{3pt}}
\AddToHook{env/introconj/after}{\addvspace{3pt}}
\AddToHook{env/remark/before}{\addvspace{3pt}}
\AddToHook{env/remark/after}{\addvspace{3pt}}
\AddToHook{env/corollary/before}{\addvspace{3pt}}
\AddToHook{env/corollary/after}{\addvspace{3pt}}
\Crefformat{enumi}{#2\textup{(#1)}#3}
\usepackage{microtype}
\usepackage{geometry}

\title{On the $p$-adic transcendence of $\sum_{k=1}^\infty p^{-1/p^k}$}

\usepackage{orcidlink}
\author{Shanwen Wang\orcidlink{0000-0003-0228-1208}}
\address{School of Mathematics, Renmin University of China, No. 59 Zhongguancun Street, Haidian District, Beijing, 100872, China}
\email{s\_wang@ruc.edu.cn}
\author{Yijun Yuan\orcidlink{0000-0001-6571-6980}}
\address{Institute for Theoretical Sciences, Westlake University, No. 600 Dunyu Road, Sandun town, Xihu district, Hangzhou, Zhejiang Province, 310030, China}
\email{941201yuan@gmail.com}
\urladdr{https://yijunyuan.github.io/}

\begin{document}
\frontmatter
\begin{abstract}
	Let $p$ be a prime number. In this article, we prove that the $p$-adic Hahn series $\sum_{k=1}^\infty p^{-1/p^k}$, which is the mixed-characteristic analogue of Abhyankar's solution $\sum_{k=1}^\infty t^{-1/p^k}$ to the Artin-Schreier equation $X^p-X-t^{-1}=0$ over $\bfF_p\pparen{t}$, is a $p$-adic complex number, but not a $p$-adic algebraic number. Based on this result, we formulate a conjecture about the possible order type of the support of an algebraic $p$-adic Hahn series and prove that it is implied by a tentative observation of Kedlaya.
\end{abstract}
\maketitle
\tableofcontents
\mainmatter
\section{Introduction}
In analogy with classical transcendental number theory, which concerns algebraicity over $\bfQ$, the study of function fields and $p$-adic fields has likewise developed into an important and well-established branch of transcendental number theory.

When $k$ is an algebraically closed field of characteristic $0$, the Puiseux-Newton theorem states that the algebraic closure of $k\pparen{t}$ is precisely the field of Puiseux series $k\pparen{t^{1/\infty}}=\bigcup_{n=1}^\infty k\pparen{t^{1/n}}$. In contrast, the structure of the algebraic closure becomes significantly more intricate when the base field $k$ has positive characteristic. For example, Chevalley showed in \cite[64]{chevalleyIntroductionTheoryAlgebraic1951} that the Artin-Schreier equation
\begin{equation}\label{eq:44045}
	X^p-X-t^{-1}=0
\end{equation}
over $\bfF_p\pparen{t}$ (and even $\overline{\bfF}_p\pparen{t}$) has no solution in the field of Puiseux series $\overline{\bfF}_p\pparen{t^{1/\infty}}$. On the other hand, Abhyankar observed in \cite{abhyankarTwoNotesFormal1956} that the formal series
\begin{equation}\label{eq:28270}
	\fraka\coloneqq \sum_{k=1}^\infty t^{-1/p^k},
\end{equation}
which lives in the field of Hahn series $\overline{\bfF}_p\pparen{t^\bfQ}$ (cf. \Cref{eg:44541}), is a root of the above equation. This example demonstrates that the framework of Hahn series offers an appropriate setting for investigating transcendental number theory over local fields of equal characteristic $p>0$. In 2001, Kedlaya gives a sufficient and necessary condition (cf. \cite[Theorem 8, Corollary 9]{kedlayaAlgebraicClosurePower2001}\footnote{See also \cite[Remark 2.9]{kedlayaAlgebraicityGeneralizedPower2017b} for the critical remark on \cite[Theorem 8]{kedlayaAlgebraicClosurePower2001}.}) for a Hahn series in $\overline{\bfF}_p\pparen{t^\bfQ}$ (resp. $\bfF_p\pparen{t^\bfQ}$) to be algebraic over $\overline{\bfF}_p\pparen{t}$ (resp. $\bfF_p\pparen{t}$) with the language of twist-recurrent sequences.

A comparable phenomenon is also observed in the mixed-characteristic context. As proved by Lampert (cf. \cite{lampertAlgebraicPadicExpansions1986}), Poonen (cf. \cite{poonenMAXIMALLYCOMPLETEFIELDS1993}) and Kedlaya (cf. \cite{kedlayaPowerSeriesPAdic2001}), for any prime number $p$, the field $\scrO_{\breve{\bfQ}_p}\pparen{p^\bfQ}$ of $p$-adic Hahn series (cf . \Cref{eg:44541}) with residue field $\overline{\bfF}_p$ and value group $\bfQ$ is the spherical completion of $\overline{\bfQ}_p$, which is algebraically closed and complete. By passing to Witt vectors, Kedlaya's criterion generalizes to verify whether a $p$-adic Hahn series lies in the \textbf{completed} integral closure of $\scrO_{\breve{\bfQ}_p}$, i.e. $\scrO_{\bfC_p}$ (cf. \Cref{coro:8415}). For example, we consider the $p$-adic analogue of Abhyankar's series \eqref{eq:28270}:
$$\frakA\coloneqq\sum_{k=1}^\infty p^{-1/p^k}\in\scrO_{\breve{\bfQ}_p}\pparen{p^\bfQ}.$$
Kedlaya's result allows us to prove
\begin{introthm}[cf. {\Cref{coro:8415}}]
	The $p$-adic Hahn series $\frakA$ lies in $\bfC_p$.
\end{introthm}

Owing to the completeness nature of Witt vectors, Kedlaya's criterion is unable to ascertain whether a $p$-adic Hahn series is algebraic over $\bfQ_p$. There are several necessaries conditions for a $p$-adic Hahn series to be algebraic over $\bfQ_p$. For example, in \cite{lampertAlgebraicPadicExpansions1986}, \cite{kedlayaPowerSeriesPAdic2001} and \cite{wang2024hyperalgebraicinvariantspadicalgebraic}, it is shown that if a $p$-adic Hahn series $\sum_{q\in\bfQ}[c_q]p^q$ is algebraic over $\bfQ_p$, then
\begin{enumerate}
	\item There exists $r\in\bfN_{\geq 1}$ such that $c_q\in\bfF_{p^r}$ for all $q\in\bfQ$;
	\item There exists $N\in\bfN_{\geq 1}$ such that the support $\{q\in\bfQ\mid c_q\neq 0\}$ is contained in $\frac{1}{N}\bfZ[p^{-1}]$;
	\item The accumulation points of the support $\{q\in\bfQ\mid c_q\neq 0\}$ are rational numbers.
\end{enumerate}
Note that the element $\frakA$ satisfies all above necessary conditions, and it is even $p$-quasi-automatic in the sense of Kedlaya \cite{kedlayaAlgebraicityGeneralizedPower2017b}. In contrast to Abhyankar's observation regarding the algebracity of $\fraka$, we establish the following result, which is counterintuitive:
\begin{introthm}[cf. {\Cref{thm:64044}}]
	The $p$-adic Hahn series $\frakA$ is transcendental over $\bfQ_p$.
\end{introthm}
This finding highlights the fundamental distinction between transcendental number theory in the context of local fields with equal characteristic $p>0$ and that pertaining to local fields of mixed characteristic $(0,p)$.

Due to the inherent difficulty in establishing a general criterion for the transcendence of $p$-adic Hahn series over $\mathbf{Q}_p$, more tractable problems are proposed. For example, Lampert asked in \cite{lampertAlgebraicPadicExpansions1986} about the possible order type of the support of an $\bfQ_p$-algebraic $p$-adic Hahn series. Inspired by the transcendence of $\frakA$, which has bounded but infinite support, we formutate the following conjecture:
\begin{introconj}[cf. {\Cref{conj:finsupp}}]
	If a $p$-adic algebraic number has bounded support, then its support is finite.
\end{introconj}
We will prove in \Cref{prop:46136} that this conjecture is implied by a hypothetical observation by Kedlaya in \cite{kedlayaPowerSeriesPAdic2001}. It is anticipated that this conjecture will constitute a foundational advancement in the study of the $\mathbf{Q}_p$-transcendence properties of $p$-adic Hahn series.


\section{Preliminaries on Hahn series}
To make this article self-contained, we briefly recall some basic facts about Hahn series.

\begin{definition}[{\cite[Section 3]{poonenMAXIMALLYCOMPLETEFIELDS1993}}]\label{def:61947}
	Let $R$ be a commutative ring and $G$ be an ordered group.
	\begin{enumerate}
		\item For any $f\in\opn{Hom}_{\opn{Set}}(G,R)$, we define the \textbf{support} of $f$ to be $$\opn{Supp}(f)=\{g\in G\colon f(g)\neq 0\}.$$
		\item Define the set of \textbf{Hahn series} over $R$ with value group $G$ to be
		      $$R\pparen{G}\coloneqq \{f\in \opn{Hom}_{\opn{Set}}(G,R)\colon \opn{Supp}(f) \text{ is well-ordered}\}.$$
		      By introducing a formal variable $t$, elements in $R\pparen{G}$ will also be written as $\sum_{g\in G}r_gt^g$, where $r_g\in R$ for all $g\in G$.
	\end{enumerate}
\end{definition}
\begin{proposition}[{\cite[Lemma 1,Corollary 2]{poonenMAXIMALLYCOMPLETEFIELDS1993}}]\label{prop:40704}
	Let $R$ be a commutative ring and $G$ be an ordered group.
	\begin{enumerate}
		\item With identity $1\cdot t^0$ and addition as well as multiplication given by
		      $$\sum_{g\in G}a_g t^g+\sum_{g\in G}b_g t^g\coloneqq\sum_{g\in G}(a_g+b_g)t^g,\ \sum_{g\in G}a_g t^g\cdot\sum_{g\in G}b_g t^g\coloneqq\sum_{g\in G}\left(\sum_{h\in G}a_h b_{g-h}\right)t^,$$
		      $R\pparen{G}$ forms a commutative ring.
		\item If $R$ is a field, then so does $R\pparen{G}$. Moreover, with the map
		      $$v\colon R\pparen{G}\lto G\cup\{\infty\},\ f\longmapsto \begin{cases}\min\opn{Supp}(f),& \text{ if }f\neq 0\\\infty,& \text{ if }f=0\end{cases},$$
		      $R\pparen{G}$ becomes a valued field with value group $G$ and residue field $R$.
	\end{enumerate}
\end{proposition}
Note that $\opn{char}R\pparen{G}=\opn{char}R$, we call $R\pparen{G}$ the \textbf{equal-characteristic field of Hahn series} over $R$ with value group $G$, also denoted as $R\pparen{t^G}$ with respect to the formal variable $t$.

\begin{proposition}[{\cite[Proposition 3, Corollary 3, Proposition 5]{poonenMAXIMALLYCOMPLETEFIELDS1993}}]\label{prop:49318}
	Let $k$ be a perfect field of characteristic $p$ and $G$ be an ordered group containing $\bbZ$ as a subgroup. Besides that, let
	$$\calN\coloneqq \left\{\sum_{g\in G}r_g t^g\in W(k)\pparen{t^G}\colon \text{ for every } g\in G,\ \sum_{n\in\bbZ}r_{g+n}p^n=0\right\},$$
	where $W(k)$ is the ring of Witt vectors of $k$. Then
	\begin{enumerate}
		\item $\calN$ is a maximal ideal of $W(k)\pparen{t^G}$, which makes $W(k)\pparen{p^G}\coloneqq W(k)\pparen{t^G}/\calN$ a field\footnote{Intuitively speaking, $W(k)\pparen{p^G}$ is obtained by replacing the formal variable $t$ of elements in $W(k)\pparen{t^G}$ by the prime $p$.}, called the \textbf{$p$-adic field of Hahn series}.
		\item Every element in $W(k)\pparen{p^G}$ can be uniquely written as $$\sum_{g\in G}[r_g]p^g,$$
		      where $r_g\in k$ for all $g\in G$ and $[\cdot]\colon k\lto W(k)$ is the Teichmüller lift. We call this the \textbf{standard expansion} of the element.
		\item For $f=\sum_{g\in G}[r_g]p^g$, define the \textbf{support} of $f$ to be
		      $$\opn{Supp}(f)=\{g\in G\colon r_g\neq 0\}.$$
		      Then the map
		      $$v\colon W(k)\pparen{G}/\calN\lto G\cup\{\infty\},\ f\mapsto \begin{cases}
				      \min\opn{Supp}(f), & \text{ if }f\neq 0 \\\infty,& \text{ if }f=0
			      \end{cases}$$
		      makes $W(k)\pparen{G}/\calN$ a mixed-characteristic valued field with value group $G$ and residue field $k$.
	\end{enumerate}
\end{proposition}

The most important property of the field of Hahn series is the following:
\begin{theorem}[cf. {\cite[Theorem 1, Corollary 4, Corollary 6]{poonenMAXIMALLYCOMPLETEFIELDS1993}}]
	Let $F$ be an equal-characteristic (resp. mixed-characteristic) valued field with divisible value group $G$ and algebraically closed residue field $k$. Then the equal-characteristic (resp. $p$-adic) field of Hahn series $k\pparen{t^G}$ (resp. $W(k)\pparen{p^G}$) is the unique (up to isomorphisms of valued field) minimal spherically complete extension of $F$. Moreover, it is algebraically closed and complete.
\end{theorem}

The following examples are the fields of Hahn series used in this article:
\begin{example}\label{eg:44541}
	Let $F=\overline{\bfF}_p\pparen{t}$ (resp. $\breve{\bfQ}_p=W(\overline{\bfF}_p)[p^{-1}]$), which has value group $\bfQ$ and residue field $\overline{\bfF}_p$. Then the field of equal-characteristic (resp. $p$-adic) Hahn series $\overline{\bfF}_p\pparen{t^\bfQ}$ (resp. $\scrO_{\breve{\bfQ}_p}\pparen{p^\bfQ}=W(\overline{\bfF}_p)\pparen{p^\bfQ}$) is the spherical completion of $F$ with residue field and value group unchanged, which is algebraically closed and complete.
\end{example}
\begin{remark}
	To simplify the statement, we will simply call $\scrO_{\breve{\bfQ}_p}\pparen{p^\bfQ}$ the $p$-adic Hahn series without specifying the residue field and value group.
\end{remark}

We end this section by proving the following lemma, which will be used in the proof of \Cref{thm:64044}:
\begin{lemma}\label{lem:12387}
	Every $p$-adic Hahn series can be uniquely written as $\sum_{q\in \bfQ\cap [0,1)}c_q\cdot p^q$, with $c_q\in \breve{\bfQ}_p$ for all $q$.
\end{lemma}
\begin{proof}
	Notice that $\bfQ\cap [0,1)$ is a set of representatives of $\bfQ/\bfZ$, every $p$-adic Hahn series $\sum_{q\in \bfQ}[a_q]p^q$ can be written as
	\begin{equation*}
		\sum_{q\in \bfQ}[a_q]p^q=\sum_{q\in \bfQ\cap[0,1)}\sum_{n\in\bfZ}[a_{q+n}]p^{q+n}=\sum_{q\in \bfQ\cap[0,1)}p^q\left(\sum_{n\in\bfZ}[a_{q+n}]p^n\right),
	\end{equation*}
	where $\sum_{n\in\bfZ}[a_{q+n}]p^n\in \breve{\bfQ}_p$ for all $q$.

	For the uniqueness, for any $p$-adic Hahn series $\sum_{q\in \bfQ}[a_q]p^q$, write
	$$\sum_{q\in \bfQ}[a_q]p^q=\sum_{q\in \bfQ\cap[0,1)}c_q\cdot p^q=\sum_{q\in \bfQ\cap[0,1)}d_q\cdot p^q$$
	with $c_q,d_q\in \breve{\bfQ}_p$ for all $q\in\bfQ\cap[0,1)$. Then we have
	$$0=\sum_{q\in\bfQ\cap[0,1)}(c_q-d_q)p^q.$$
	If we write $c_q-d_q=\sum_{n\in\bfZ}[s_{q,n}]p^n\in\breve{\bfQ}_p$ with $s_{q,n}\in \overline{\bfF}_p$ for all $n$ and for all $q\in\bfQ\cap[0,1)$, then
	\begin{equation}\label{eq:19959}
		0=\sum_{q\in\bfQ\cap[0,1)}\sum_{n\in\bfZ}[s_{q,n}]p^{q+n}.
	\end{equation}
	Since $q+n$ are all distinct for different pairs $(q,n)$, \Cref{eq:19959} is the standard expansion of $0\in\scrO_{\breve{\bfQ}_p}\pparen{p^\bfQ}$. Hence $s_{q,n}=0$ for all $q\in\bfQ\cap[0,1)$ and $n\in\bfZ$, i.e. $c_q=d_q$ for all $q\in\bfQ\cap[0,1)$.
\end{proof}

\section{$\frakA$ is a $p$-adic complex number}
Although the summation $\sum_{k=1}^\infty p^{-1/p^k}$ is not interpreted as the $p$-adic limit (which does not exist) of the suspicious sequence $\left\{\sum_{k=1}^n p^{-1/p^k}\right\}_{n\geq 1}\subset \overline{\bfQ}_p$, we can still prove that it is a $p$-adic limit of a sequence in $\overline{\bfQ}_p$. This is a direct consequence of the following result of Kedlaya:
\begin{theorem}[cf. {\cite[Theorem 13.4]{kedlayaAlgebraicityGeneralizedPower2017b}}]\label{coro:8415}
	The ring $\scrO_{\bfC_p}$ equals to the following sets:
	\begin{enumerate}
		\item the completion of the image of the $p$-quasi-automatic elements (cf. \cite[Definition 6.3, Definition 13.1]{kedlayaAlgebraicityGeneralizedPower2017b}) of $W\left(\overline{\bfF}_p\right)\pparen{t^\bfQ}^\wedge$ under the projection
		      $$W\left(\overline{\bfF}_p\right)\pparen{t^\bfQ}^\wedge\lto \scrO_{\breve{\bfQ}_p}\pparen{p^\bfQ}.$$
		\item the completion of the following set:
		      $$\left\{\sum_{q\in\bfQ}[c_q]p^q\in \scrO_{\breve{\bfQ}_p}\pparen{p^\bfQ}\middle\vert \sum_{q\in\bfQ}c_q\cdot t^q\in \widehat{\scrO}_{L}\right\},$$
		      where $\widehat{\scrO}_{L}$ is the completion of the integral closure of $\overline{\bfF}_p\bbrac{t}$ in $\overline{\bfF}_p\pparen{t^\bfQ}$.
	\end{enumerate}
\end{theorem}
\begin{remark}
	In the original statement of \cite[Theorem 13.4]{kedlayaAlgebraicityGeneralizedPower2017b}, the ring $\widehat{\scrO}_L$ (resp. $\scrO_{\bfC_p}$) is described as the ``completed integral closure'' of the field $\overline{\bfF}_p\pparen{t}$ (resp. $W\left(\overline{\bfF}_p\right)[p^{-1}]=\breve{\bfQ}_p$). By comparing \cite[Theorem 11.12]{kedlayaAlgebraicityGeneralizedPower2017b} with \cite[Definition 6.3]{kedlayaAlgebraicityGeneralizedPower2017b}, we believe that the ``completed integral closure'' of a valued field $F$ in a larger valued field $E$ in \cite{kedlayaAlgebraicityGeneralizedPower2017b} means the completion of the integral closure of $\scrO_F$ in $E$, rather than the completion of the integral closure (which is also the algebraic closure) of $F$ in $E$, which is the literal interpretation of ``completed integral closure''.
\end{remark}

As an application, we can prove that $\frakA$ can be approximated by a Cauchy sequence in $\overline{\bfQ}_p$:
\begin{proposition}\label{prop:57046}
	The $p$-adic Hahn series $\frakA=\sum_{k=1}^\infty p^{-1/p^k}$ lies in $\bfC_p$.
\end{proposition}
\begin{proof}
	One can verify by direct calculation that the Hahn series $\sum_{k=1}^\infty t^{1-1/p^k}$ is a root of the polynomial
	$$X^p-t^{p-1}X-t^{p-1}\in \overline{\bfF}_p\bbrac{t}[X].$$
	By \Cref{coro:8415}, we conclude that $p\cdot\frakA=\sum_{k=1}^\infty p^{1-1/p^k}$ lies in $\bfC_p$, and the result follows.
\end{proof}
\begin{remark}
	The proof of \Cref{prop:57046} is not effective, i.e. it does not provide an explicit Cauchy sequence in $\overline{\bfQ}_p$ that converges to $\frakA$. By Kedlaya's transfinite Newton algorithm (cf. \cite[Theorem 1]{kedlayaPowerSeriesPAdic2001}, \cite[Section 2.2]{wangUniformizerFalseTate2021}), we mention two examples of $p$-adic algebraic numbers, whose $p$-adic distance to $\frakA$ is less than $1$:
	\begin{enumerate}
		\item  There exists a root $\alpha$ of the Artin-Schreier polynomial $X^p-X-p^{-1}\in \bfQ_p[X]$, which is the $p$-adic analogue of \Cref{eq:44045}, such that
		      $$\alpha= \sum_{k=1}^\infty p^{-1/p^k}+p^{1/p-1/p^2}+\text{ terms with higher valuation},$$
		      i.e. $v_p(\alpha-\frakA)=\frac{1}{p}-\frac{1}{p^2}$.
		\item When $p\geq 3$, then \cite{wangTruncatedExpansion} shows that for any integer $n\geq 2$, there exists a $p^n$-th root of unity $\zeta_{p^n}$ with expansion
		      \begin{align*}
			      \zeta_{p^n}= & \sum_{k=0}^{p-1}\frac{\zeta_{2(p-1)}^k}{[k!]}p^{\frac{k}{p^{n-1}(p-1)}} + \zeta_{2(p-1)}p^{\frac{1}{p^{n-2}(p-1)}}\sum_{k=n}^\infty p^{-1/p^k} \\
			                   & \quad +\zeta_{2(p-1)}^2p^{\frac{1}{p^{n-2}(p-1)}+\frac{1}{p^n(p-1)}} +\text{ terms with higher valuation},
		      \end{align*}
		      where $\zeta_{2(p-1)}$ is a fixed primitive $2(p-1)$-th root of unity.
		      If we set
		      $$\beta_n\coloneqq \left(\zeta_{2(p-1)}^2\cdot p^{\frac{1}{p^{n-2}(p-1)}}\right)^{-1}\cdot \left(\zeta_{p^n}-\sum_{k=0}^{p-1}\frac{\zeta_{2(p-1)}^k}{[k!]}p^{\frac{k}{p^{n-1}(p-1)}}\right)+\sum_{k=1}^{n-1}p^{-1/p^k},$$
		      then
		      $$v_p(\beta_n-\frakA)=\frac{1}{p^n(p-1)}.$$
	\end{enumerate}
\end{remark}

\section{$\bfQ_p$-transcendence of $\frakA$}
To show that $\frakA$ is transcendental over $\bfQ_p$, we prove by contradiction by assuming that $\frakA$ is the root of a polynomial $f$ with $\bfZ_p$ coefficients. Our strategy is straightforward: we are going to show that there are some terms in the multinomial expansion of the leanding term of $f(\frakA)$ is impossible to be cancelled out by any other terms in the whole expansion. To rigiously carry out this idea, we need some preparations.

\begin{definition}Let $\bbI\coloneqq \bigoplus_{\bfN_{>0}}\bfN$.
	\begin{enumerate}
		\item Let $\lambda\colon \bbI\lto \bfQ,\ (a_k)_{k\geq 1}\longmapsto \sum_{k=1}^\infty -\frac{a_k}{p^k}$.
		\item Let $\Sigma\colon \bbI\lto \bfN,\ (a_k)_{k\geq 1}\longmapsto \sum_{k=1}^\infty a_k$.
		\item Let $\kappa\colon \bbI\lto \bfN,\ (a_k)_{k\geq 1}\longmapsto \max\{k\geq 1\mid x_k>p-1\}$, where we set $\max\emptyset=0$. This always makes sense since for every $\underline{a}=(a_k)_{k\geq 1}\in\bbI$, there are only finitely many $k$'s with $a_k\neq 0$.
		\item We call an element $\underline{a}\in\bbI$ \textbf{reduced}, if $\kappa(\underline{a})=0$, i.e. $0\leq a_k\leq p-1$ for all $k\geq 1$.
		\item For any $\underline{a},\underline{b}\in \bbI$, write $\underline{a}\sim \underline{b}$ if $\lambda(\underline{a})-\lambda(\underline{b})\in\bfZ$. It is easy to see that $\sim$ is an equivalence relation.
	\end{enumerate}
\end{definition}
With above notations, we can write the multinomial expansion of $\frakA^i$, $i=0,\cdots,n+1$ as
\begin{equation}
	\frakA^i=\sum_{\substack{\underline{k}\in \bbI\\\Sigma(\underline{k})=i}}s_i\binom{i}{\underline{k}}p^{\lambda(\underline{k})}.
\end{equation}
Intuitively, the ``cancellation'' could only happen between terms $p^{\lambda(\underline{a})}$ and $p^{\lambda(\underline{b})}$ with $\underline{a}\sim \underline{b}$. Hence we need to understand the equivalence classes of $\sim$ in $\bbI$. The following lemma shows that each equivalence class contains a unique reduced element, which plays the role of the ``canonical representative'' of the equivalence class.

\begin{lemma}\label{lem:47600}\leavevmode
	\begin{enumerate}
		\item For every $\underline{a}\in \bbI$, there exists a unique reduced element $\underline{a}_{\opn{red}}\in \bbI$ such that $\underline{a}\sim \underline{a}_{\opn{red}}$.
		\item One has $\Sigma(\underline{a}_{\opn{red}})\leq \Sigma(\underline{a})$. The equality holds if and only if $\underline{a}$ is reduced, i.e. $\underline{a}=\underline{a}_{\opn{red}}$.
	\end{enumerate}
\end{lemma}
\begin{proof}
	We prove by induction on $\kappa(\underline{a})$ that there exists a reduced element $\underline{a}_{\opn{red}}\in \bbI$ such that $\underline{a}\sim \underline{a}_{\opn{red}}$ and $\Sigma(\underline{a}_{\opn{red}})\leq \Sigma(\underline{a})$.

	If $\kappa(\underline{a})=0$, then we can take $\underline{a}_{\opn{red}}=\underline{a}$. Suppose that $\kappa(\underline{a})=n+1$ for certain $n\in\bfN$ and the claim holds for all elements in $\bbI$ with $\kappa$-value not greater than $n$.

	Since $a_{n+1}>p-1$, we can write $a_{n+1}=r+p\cdot d$, with $r\in \{0,\cdots,p-1\}$ and $d\in\bfN_{>0}$. Let $\underline{a}'\coloneqq (a_k')_{k\geq 1}\in \bbI$ be defined as follows:
	$$a_k'\coloneqq \begin{cases}
			a_k,   & \text{ if } k\neq n,n+1;               \\
			r,     & \text{ if } k=n+1;                     \\
			a_n+d, & \text{ if } n \geq 1 \text{ and } k=n.
		\end{cases}$$
	Then one can check by direct calculation that
	$$\lambda(\underline{a})=-\frac{r+p\cdot d}{p^{n+1}}+\sum_{k\neq n+1}-\frac{a_k}{p^k}=-\frac{r}{p^{n+1}}-\frac{a_n+d}{p^n}+\sum_{k\neq n, n+1}-\frac{a_k}{p^k}=\lambda(\underline{a}')$$
	when $n\geq 1$, and
	$$\lambda(\underline{a})=-\frac{r+p\cdot d}{p}+\sum_{k\geq 2}-\frac{a_k}{p^k}=\lambda(\underline{a}')-d$$
	when $n=0$. In both cases, we have $\underline{a}\sim \underline{a}'$ and $\Sigma(\underline{a}')\leq n$. By the induction hypothesis, there exists a reduced element $\underline{a}_{\opn{red}}'\in \bbI$ such that $\underline{a}'\sim \underline{a}_{\opn{red}}'$ and $\Sigma(\underline{a}_{\opn{red}}')\leq \Sigma(\underline{a}')$. Hence we have $\underline{a}\sim \underline{a}_{\opn{red}}'$ and $\Sigma(\underline{a}_{\opn{red}}')\leq \Sigma(\underline{a})$. This finishes the proof of the existence.

	Now take two reduced elements $\underline{x}$ and $\underline{y}$ with $\underline{x}\sim\underline{y}$. Since
	$$-1=-\sum_{k=1}^{\infty}\frac{p-1}{p^k}< \lambda(\underline{x}), \lambda(\underline{y})\leq 0,$$
	the condition $\underline{x}\sim \underline{y}$ implies that $\lambda(\underline{x})=\lambda(\underline{y})$. By viewing $\lambda(\underline{x})$ and $\lambda(\underline{y})$ as floating point numbers in base $p$, a digit-by-digit comparison shows that they are equal. This finishes the proof of the uniqueness.

	For the last statement, suppose that $\Sigma(\underline{a}_{\opn{red}})=\Sigma(\underline{a})$. If $\Sigma(\underline{a})=0$, then $\underline{a}=\underline{a}_{\opn{red}}=(0,0,\cdots)$. Now we assume $\Sigma(\underline{a})>0$ and $\underline{a}$ is not reduced. Then the construction in the first part of the proof shows that there exists $\underline{a}'\in\bbI$ with $\underline{a}'\sim\underline{a}$ and $\Sigma(\underline{a}')<\Sigma(\underline{a})$. As a result, we obtain $\Sigma(\underline{a}_{\opn{red}})\leq \Sigma(\underline{a}')<\Sigma(\underline{a})$, which contradicts the assumption.
\end{proof}

\begin{corollary}
	For every $\underline{a},\underline{b}\in \bbI$, $\underline{a}\sim \underline{b}$ if and only if $\underline{a}_{\opn{red}}=\underline{b}_{\opn{red}}$.
\end{corollary}

Now we are ready to prove the main result of this article:
\begin{theorem}\label{thm:64044}
	The $p$-adic Hahn series $\frakA=\sum_{k=1}^\infty p^{-1/p^k}$ is transcendental over $\bfQ_p$.
\end{theorem}
\begin{proof}
	Suppose the contrary that there exists $n\in\bfN$ and $s_0,\cdots,s_{n+1}\in\bfZ_p$ such that
	$$s_0+s_1\frakA+\cdots+s_n\frakA^n+s_{n+1}\frakA^{n+1}=0,$$
	with $s_0$ and $s_{n+1}$ nonzero. The multinomial theorem implies that
	$$0=\sum_{i=0}^{n+1}s_i\frakA^i=\sum_{i=0}^{n+1}\sum_{\substack{\underline{k}\in \bbI\\\Sigma(\underline{k})=i}}s_i\binom{i}{\underline{k}}p^{\lambda(\underline{k})}=\sum_{\substack{\underline{k}\in \bbI\\\Sigma(\underline{k})\leq n+1}}s_{\Sigma(\underline{k})}\binom{\Sigma(\underline{k})}{\underline{k}}p^{\lambda(\underline{k})}.$$
	If we group the terms on the right hand side according to the equivalence classes of $\sim$, we obtain
	\begin{equation}\label{eq:7708}
		\begin{aligned}
			0= & \sum_{\substack{\underline{k}_{\opn{red}}\in\bbI \\\underline{k}_{\opn{red}} \text{ reduced}}}\left(\sum_{\substack{\underline{k}\in\bbI\\\Sigma(\underline{k})\leq n+1\\\underline{k}\sim\underline{k}_{\opn{red}}}}s_{\Sigma(\underline{k})}\binom{\Sigma(\underline{k})}{\underline{k}}p^{\lambda(\underline{k})}\right)\\
			=  & \sum_{\substack{\underline{k}_{\opn{red}}\in\bbI \\\underline{k}_{\opn{red}} \text{ reduced}}}p^{\lambda(\underline{k}_{\opn{red}})+\delta_0}\left(\sum_{\substack{\underline{k}\in\bbI\\\Sigma(\underline{k})\leq n+1\\\underline{k}\sim\underline{k}_{\opn{red}}}}s_{\Sigma(\underline{k})}\binom{\Sigma(\underline{k})}{\underline{k}}p^{\lambda(\underline{k})-\lambda(\underline{k}_{\opn{red}})-\delta_0}\right),
		\end{aligned}
	\end{equation}
	where $$\delta_0=\begin{cases}
			0, & \text{ if } \underline{k}_{\opn{red}}=(0,0,\cdots); \\
			1, & \text{ otherwise}.
		\end{cases}$$

	Observe that for different reduced elements $\underline{k}_{\opn{red}}\in\bbI$, the values $\lambda(\underline{k}_{\opn{red}})+\delta_0\in[0,1)$ are all distinct, and if $\underline{k}\sim\underline{k}_{\opn{red}}$, then $\lambda(\underline{k})-\lambda(\underline{k}_{\opn{red}})-\delta_0\in\bfN$. Hence \Cref{lem:12387} allows us to conclude from \Cref{eq:7708} that
	\begin{equation}\label{eq:50446}
		\sum_{\substack{\underline{k}\in\bbI\\\Sigma(\underline{k})\leq n+1\\\underline{k}\sim\underline{k}_{\opn{red}}}}s_{\Sigma(\underline{k})}\binom{\Sigma(\underline{k})}{\underline{k}}p^{\lambda(\underline{k})-\lambda(\underline{k}_{\opn{red}})}=0
	\end{equation}
	for every reduced element $\underline{k}_{\opn{red}}\in\bbI$. If we take the reduced element
	$$\underline{k}^*\coloneqq (\overbrace{1,1,\cdots,1}^{n+1 \text{ times}},0,\cdots,)\in\bbI,$$
	then for every $\underline{k}\in\bbI$ with $\Sigma(\underline{k})\leq n+1$ and $\underline{k}\sim \underline{k}^*$, \Cref{lem:47600} implies that $\Sigma(\underline{k})=\Sigma(\underline{k}^*)=n+1$ and consequently $\underline{k}=\underline{k}^*$. Hence \Cref{eq:50446} can be specialized to the case of $\underline{k}_{\opn{red}}=\underline{k}^*$ as
	$$s_{n+1}\binom{n+1}{\underline{k}^*}=s_{n+1}\cdot (n+1)!=0,$$
	which contradicts the assumption that $s_{n+1}$ is nonzero.
\end{proof}
\begin{remark}
	One does not need to worry about the well-definedness of any of the infinite sums above, since for every exponent with value of the form $\lambda(\underline{k})$, there are only finitely many terms contributing to it (cf. \cite[Lemma 1]{poonenMAXIMALLYCOMPLETEFIELDS1993}).
\end{remark}

\begin{remark}
	After some necessary generalization of $\lambda$, $\Sigma$ and $\kappa$, the same idea can potentially be applied to show the $\bfQ_p$-transcendence of more general $p$-adic Hahn series.
\end{remark}

\section{Order type of $p$-adic algebraic numbers}
In \cite{lampertAlgebraicPadicExpansions1986}, Lampert asked about the possible order type of the support of the expansion of $p$-adic algebraic numbers as $p$-adic Hahn series. Kedlaya showed in \cite[Section 4]{kedlayaPowerSeriesPAdic2001} that the order type can not exceed $\omega^\omega$. He also made the following remark at the end of same section:
\begin{framed}
	\texttt{ In fact, it is entirely possible that the answer to Lampert's question is that only finite orders, $\omega$ and $\omega^\omega$ can occur.}
\end{framed}
We simply call it \textbf{Kedlaya's prediction}. Note that this prediction indicates that there does not exist any $p$-adic algebraic number with bounded support of order type $\omega$: if $\alpha=\sum_{k=0}^\infty[c_k]p^{r_k}\in \scrO_{\breve{\bfQ}_p}\pparen{p^\bfQ}$ is a $p$-adic algebraic number with $\opn{diam}(\opn{Supp}(\alpha))<N$ for certain positive integer $N$, then
$$\frac{\alpha}{1-p^N}=\sum_{k=0}^\infty[c_k]p^{r_k}+\sum_{k=0}^\infty[c_k]p^{r_k+N}+\sum_{k=0}^\infty[c_k]p^{r_k+2N}+\cdots$$
is a $p$-adic algebraic number with order type $\omega^2$, which contradicts Kedlaya's prediction.

Since currently no $p$-adic algebraic number with bounded support of infinite order type is known, and even simple $p$-adic Hahn series like $\frakA$ is transcendental over $\bfQ_p$, we raise the following conjecture about the order type of $p$-adic algebraic numbers:
\begin{conjecture}\label{conj:finsupp}
	If a $p$-adic algebraic number has bounded support, then its support is finite.
\end{conjecture}
In fact, this conjecture is also implicated by Kedlaya's prediction:
\begin{proposition}\label{prop:46136}
	Kedlaya's prediction implies \Cref{conj:finsupp}.
\end{proposition}
\begin{proof}
	Suppose $x$ is a $p$-adic algebraic number with bounded support of order type $\alpha$. Let $N$ be a positive integer such that $\opn{diam}(\opn{Supp}(x))<N$. Similar to the argument above, we can show that
	$$\frac{x}{1-p^N}=\sum_{t=0}^\infty x\cdot p^{Nt}$$
	is a $p$-adic algebraic number with order type $\alpha\cdot \omega$. If we write $\alpha$ as its Cantor normal form (cf. \cite[Chapter XIV, §19, Theorem 2]{zbMATH03212886})
	$$\alpha=\omega ^{\beta _{1}}c_{1}+\omega ^{\beta _{2}}c_{2}+\cdots +\omega ^{\beta _{k}}c_{k},$$
	where $\beta_1>\beta_2>\cdots>\beta_k\geq 0$ are ordinals and $c_i$ are positive integers, then the Cantor normal form of $\alpha\cdot\omega$ is $\omega^{\beta_1+1}$ (cf. \cite[Chapter XIV, §19, Exercise 4]{zbMATH03212886}).
	By Kedlaya's prediction, we split the proof into 3 cases:
	\begin{enumerate}
		\item If $\alpha\cdot\omega$ is finite, then clearly $\alpha=0$;
		\item If $\alpha\cdot \omega=\omega$, then $\beta_1=0$, and consequently $\alpha=c_1$ is finite;
		\item If $\alpha\cdot \omega=\omega^\omega$, then $\omega^{\beta_1+1}=\omega^\omega$, which is impossible since $\omega$ is not a successor ordinal.
	\end{enumerate}
\end{proof}

\backmatter
\printbibliography
\end{document}